\documentclass[a4, 11pt]{amsart}
\usepackage{amssymb,latexsym}
\usepackage{color}
\usepackage[left=35mm, right=35mm, top=30mm, bottom=30mm, includefoot, includehead]{geometry}
\theoremstyle{plain}
\newtheorem{theorem}{Theorem}[section]
\newtheorem{corollary}[theorem]{Corollary}
\newtheorem{proposition}[theorem]{Proposition}
\newtheorem{lemma}[theorem]{Lemma}
\theoremstyle{definition}
\newtheorem{definition}[theorem]{Definition}
\newtheorem{example}[theorem]{Example}

\newtheorem{remark}[theorem]{Remark}

\DeclareMathOperator{\gen}{gen}
\DeclareMathOperator{\red}{red}

\DeclareMathOperator{\reg}{reg}

\newcommand{\enm}[1]{\ensuremath{#1}}          %

\newcommand{\cal}[1]{\mathcal{#1}}

\renewcommand{\bar}[1]{\overline{#1}}

\newcommand{\CC}{\enm{\mathbb{C}}}

\newcommand{\NN}{\enm{\mathbb{N}}}
\newcommand{\RR}{\enm{\mathbb{R}}}

\newcommand{\PP}{\enm{\mathbb{P}}}
\newcommand{\LL}{\enm{\mathbb{L}}}

\newcommand{\KK}{\enm{\mathbb{K}}}

\newcommand{\Cc}{\enm{\cal{C}}}

\newcommand{\Ii}{\enm{\cal{I}}}

\newcommand{\Ss}{\enm{\cal{S}}}

\renewcommand{\phi}{\varphi}
\renewcommand{\theta}{\vartheta}
\renewcommand{\epsilon}{\varepsilon}

\renewcommand{\to}[1][]{\xrightarrow{\ #1\ }}

\newcommand{\old}[1]{}

\date{}
\title{Labels of real projective varieties}

\author{Edoardo Ballico and Emanuele Ventura}

\address{Universit\`a di Trento, 38123 Povo (TN), Italy}
\email{edoardo.ballico@unitn.it}
\address{Dept. of Mathematics, Texas A\&M University,
College Station, TX 77843-3368, USA}
\email{eventura@math.tamu.edu, emanueleventura.sw@gmail.com}


\begin{document}

\maketitle

\begin{abstract}

Let $X$ be a complex projective variety defined over $\mathbb R$. Recently, Bernardi and the first author introduced the notion of admissible rank with respect to $X$. This rank takes into account only decompositions that are stable under complex conjugation. Such a decomposition carries a label, i.e., a pair of integers recording the cardinality of its totally real part. We study basic properties of admissible ranks for varieties, along with special examples of curves; for instance, for rational normal curves admissible and complex ranks coincide. Along the way, we introduce the scheme theoretic version of admissible rank. 
Finally, analogously to the situation of real ranks, we analyze typical labels, i.e., those arising as labels of a full-dimensional Euclidean open set. We highlight similarities and differences with typical ranks. 
\end{abstract}

\section{Introduction}

Given a projective variety $X\subset \PP^r$ defined over an arbitrary field, one defines the $X$-rank for every point in its ambient space. 
This notion generalizes the more classical tensor ranks and Waring ranks, when $X$ is a Segre and a Veronese variety, respectively. 
Nowadays, ranks show up very frequently in several different applications ranging from algebraic complexity to quantum information theory \cite{Lands}. 

Often ranks are defined over the complex or the real numbers, where quite different and interesting phenomena appear. Following the recent works \cite{bb, bal}, here we consider an intermediate rank between the two, the so-called {\it admissible rank}. The first datum is a complex projective variety $X \subset \PP^r$ that is {\it defined over $\RR$}. Informally, this is equivalent to $X$ being defined by finitely many homogeneous polynomials with real coefficients. For any real point $q\in \PP^r(\RR)$, Bernardi and the first author \cite{bb} defined the admissible rank of $q$ with respect to $X$ to be the smallest finite subset $S\subset X$ that is stable under complex conjugation. Such 
a stable set $S$ carries a label (Definition \ref{basic}), i.e., a pair of integers recording the structure of $S$. Whence our aim is to investigate what are the labels that can possibly arise from $X$ and which fundamental properties they possess. Perhaps the most interesting among these labels are those arising in open Euclidean subsets of the ambient real projective space $\PP^r(\RR)$, the {\it typical labels}; see \S 3 and \S 4. We highlight differences and similarities that we find between typical labels and typical ranks; the latter ones have been intensively studied in recent years; see for instance \cite{abc, bb, bbo, b, BrSt, mm,mmsv, v}. 

Besides being interesting on its own, we believe the study of admissible ranks could shed light on questions about $X$-ranks over the real or complex numbers without any sort of restrictions on the decompositions one looks for.  \\

\noindent {\bf Contributions and structure of the paper.} In \S 2, we first discuss our setting and basic definitions. Proposition \ref{eo2} recalls an upper bound on the admissible rank of any point with respect to $X(\CC)\subset \PP^r(\CC)$; interestingly, the same bound holds over $\CC$ but is known to fail over $\RR$. In Proposition \ref{oo1}, we show an analogous result for complex space curves $X(\CC)$ equipped with a real structure such that $X(\RR)\neq \emptyset$. 
In Proposition \ref{upperbound2g}, we derive a {\it Blekherman-Teitler type} upper bound on the maximum admissible rank; this is the same as the one for maximum complex rank. 

As for usual ranks, we introduce a scheme-theoretic version of admissible rank, the {\it admissible cactus rank} (Definition \ref{labelcactus}). Along the same lines, we define the {\it scheme-label} of a zero-dimensional scheme (Definition \ref{labelforscheme}). We study scheme-labels under a specific regime described in Remark \ref{e0}; see Proposition \ref{e1}. 

In \S 3, we focus on {\it typical labels}, i.e., those arising as labels of non-empty Euclidean subsets in the ambient space. Alike typical ranks, Proposition \ref{typicallabels} establishes that all labels $(a,b)$ with weight $2a+b=g$, where $g$ is the complex generic rank of $X(\CC)\subset \PP^r(\CC)$, are typical. However, there exist instances where a typical label $(a,b)$ might have weight $2a+b>g$, as Theorem \ref{v1} shows:
a linearly normal elliptic curve defined over $\RR$ provides such an example. On the contrary, under a numerical assumption on degree and arithmetic genus, Theorem \ref{w11} states that an integral nondegenerate space curve with only planar singularities has no typical label $(a,b)$ with weight $2a+b>2=g$. 

\S 4 is devoted to rational normal curves. Theorem \ref{o1} shows that admissible rank and complex rank coincide in this case. Moreover, 
cactus and admissible cactus rank are equal as well. A corollary of this result, already noticed in \cite{bb}, is that typical labels for the rational normal curve
are only those of the form $(a,b)$ with $2a+b = g$, where $g$ is its complex generic rank (Corollary \ref{o2}). Theorem \ref{o3} shows that, albeit the linearly normal embedding of a plane conic $\mathcal C = \lbrace x^2+y^2+z^2=0 \rbrace\subset \PP^2(\CC)$ has no real points, it shares the same rank properties of the usual rational normal curve. 

We introduce the notion of {\it admissible rank boundaries} in Definition \ref{boundary}, following the real case; see e.g. \cite{mmsv,BrSt}. In Remark \ref{component}, we find an irreducible component of this boundary for rational normal curves. Example \ref{twolabels} shows the existence of a point having {\it minimal}
admissible decompositions with distinct labels. 

Finally, in \S 5, we introduce {\it real joins} of semialgebraic sets and prove Theorem \ref{eu1}; this may be viewed as a (weaker) version of \cite[Theorem 3.1]{bhmt} over the reals. 

\vspace{3mm}
\begin{small}
\noindent {\bf Acknowledgements.} The first author was partially supported by MIUR and GNSAGA of INdAM (Italy). The second author would like to thank the Department of Mathematics of Universit\`{a} di Trento for the warm hospitality. 
\end{small}

\section{Preliminaries}
Let $X$ be a projective variety defined over an arbitrary field $\KK$. For a subfield
$\LL \subseteq \KK$, the set of $\LL$-points of $X$ is denoted $X(\LL)$. The set of smooth $\LL$-points of $X$ 
is $X_{\reg}(\LL)$ and the singular locus $X_{\textnormal{sing}}(\LL)$. In the following, $\KK = \CC$ and $\LL = \RR$.

Let $X(\CC)\subset \PP^r(\CC)$ be an integral and nondegenerate complex projective variety.
The pair consisting of $X(\CC)$ and the given embedding $X(\CC )\hookrightarrow \PP^r(\CC)$ is {\it defined over} $\RR$ if and only if
$\sigma (X(\CC)) = X(\CC)$, where $\sigma: \PP^r(\CC) \to \PP^r(\CC)$ is the usual complex conjugation map $[z_0:\dots :z_r]\mapsto
[\bar{z_0}:\dots :\bar{z_r}]$. Define their {\it totally real} parts to be: 
\[
\PP^r(\RR) =\{z\in \PP^r(\CC)\mid \sigma (z)=z\} \mbox{ and } X(\RR) = X(\CC)\cap \PP^r(\RR) =\{z\in X(\CC)\mid \sigma (z)=z\}.
\]

More generally, for any closed subscheme $Z\subseteq \PP^r(\CC)$, the pair consisting of $Z$ and the given embedding is defined over $\RR$  if and only if its defining homogeneous ideal $I_Z$ may be generated by homogeneous polynomials with real coefficients only. 

Following the recent works \cite{bal, bb}, a \emph{label} is a pair $(a,b)\in \NN^2\setminus \{(0,0)\}$. The {\it weight} of a label
$(a,b)$ is the integer $2a+b$. The cardinality of a finite set $S$ is $\sharp (S)$.

\begin{definition}\label{basic}
A finite set $S \subset X(\CC)$, $S\ne \emptyset$, is said to have a \emph{label} if $\sigma (S)=S$. 
If $S$ has a label and the cardinality of its totally real part is $b =\sharp (S\cap X(\RR ))$, then the label of $S$ is $(\left(\sharp (S)-b\right)/2,b)$. 
\end{definition}

For any finite subset $S\subset X(\CC)$, let $\langle S\rangle _{\CC}$ denote the minimal complex linear subspace
of $\PP^r(\CC)$ containing $S$. Set $\langle S\rangle _{\RR}= \langle S\rangle _{\CC}\cap \PP^r(\RR)$. If $\sigma (S) = S$, then one has
$\sigma (\langle S\rangle _{\CC}) =\langle S\rangle _{\CC}$ and hence $\dim _{\CC}\langle S\rangle _{\CC}
= \dim _{\RR}\langle S\rangle _{\RR}$. 

\begin{definition}
Let $X(\KK)$ be a projective irreducible nondegenerate variety over a field $\KK$. The $X(\KK)$-rank of a given $p\in \PP^r(\KK)$, denoted
$r_{X(\KK)}(p)$, is the minimal cardinality of a subset $S\subset X(\KK)$ such that $p\in \langle S\rangle _{\KK}$. 
\end{definition}

For any integer $k>0$, the $k$-th secant variety $\sigma _k(X(\CC ))$ of $X(\CC)$
is the closure in $\PP^r(\CC)$ of the union of all the linear spaces $\langle S\rangle _{\CC}$ with $S\subset X(\CC)$ and $\sharp(S)=k$. If
$X(\CC)\subset \PP^r(\CC)$ is defined over $\RR$, then the variety $\sigma _k(X(\CC ))$ is defined over $\RR$ and
$\sigma _k(X(\CC))\cap \PP^r(\RR)$ is the set $\sigma _k(X(\CC ))(\RR)$ of the real points of $\sigma _k(X(\CC ))$. 
Often $\sigma _k(X(\CC ))(\RR)$ is bigger than the set $\sigma _k(X(\RR))$, i.e., the closure in
$\PP^r(\RR)$ of all the points with $X(\RR)$-rank $k$. In other words, a linear combination of complex points could be real. 
 
\begin{definition}[{\bf \cite[Definition 13]{bb}}]\label{basic1}
Let $X(\CC)\subset \PP^r(\CC)$ be a projective irreducible nondegenerate variety defined over $\RR$. For any $q\in \PP^r(\RR)$, the {\it admissible rank} of $q$, denoted $\ell _{X(\CC),\sigma}(q)$, is the minimal cardinality of a finite subset $S\subset X(\CC)$ such that $\sigma (S)=S$ and $q\in \langle S\rangle _{\RR}$.
Let $\Ss (X(\CC), q)$ denote the set of all $S\subset X(\CC)$ such that $\sigma (S)=S$,
$\sharp (S) =\ell _{X(\CC),\sigma}(q)$ and $q\in \langle S\rangle _{\RR}$. 
\end{definition}

From Definition \ref{basic1}, we have the next observation: 

\begin{remark}\label{a1}
For any $q\in \PP^r(\RR)$, we have 
\[
r_{X(\CC)}(q) \le \ell _{X(\CC),\sigma}(q) \le 2r_{X(\CC)}(q).
\]
One has $\ell _{X(\CC),\sigma}(q) = r_{X(\RR)}(q)$ if and only if
$(0,\ell _{X(\CC),\sigma})$ is a label of a decomposition of $q$ in $\Ss (X(\CC), q)$, i.e., it is of {\it minimal weight}. 
\end{remark}

\begin{lemma}\label{a3}
Let $X(\CC)\subset \PP^r(\CC)$ be an integral and nondegenerate complex projective variety defined over $\RR$ with $n = \dim
X(\CC)$ and $d = \deg (X(\CC))$. Fix $q\in \PP^r(\RR)\setminus X(\RR)$.
\begin{enumerate}
\item[(i)] If $n\ge 2$, there is a hyperplane $H\subset \PP^r(\CC)$ defined over $\RR$, with $q\in H$, such that the intersection $Y(\CC)= X(\CC)\cap
H$ is integral of degree $d$, spans $H$ and $Y_{\reg}(\CC)  = X_{\reg}(\CC) \cap H$. Moreover $Y(\CC)$ is defined over $\RR$ and $Y(\RR) = X(\RR)\cap H$;

\item[(ii)] If $n=1$, there is a hyperplane $H\subset \PP^r(\CC)$ defined over $\RR$, with $q\in H$ and $S= X(\CC)\cap H$
consisting of $d$ distinct points, $\sigma (S)=S$ and $\langle S\rangle _{\CC} =H$.
\end{enumerate}
\end{lemma}

\begin{proof}
(i). Let $\Pi (\CC)$ be the set of all hyperplanes $M\subset \PP^r(\CC)$ containing $q$. Thus $\Pi(\CC)$ is an $(r-1)$-dimensional
projective space. Since $q\in \PP^r(\RR)$, $\Pi (\CC)$ is defined over $\RR$ and $\Pi (\RR)$ is an $(r-1)$-dimensional real projective
space. By Bertini's theorem, there is a non-empty Zariski open subset $V \subset \Pi (\CC)$ such that for all $H\in V$ the algebraic set $X_{\reg}(\CC) \cap H$ is smooth and of dimension $\dim X(\CC) -1$, it is irreducible if $\dim X(\CC)>1$
and $\dim (X(\CC)\cap H)_{\textnormal{sing}} = \dim (X_{\textnormal{sing}}(\CC)) -1$ (with the convention $\dim(\emptyset)=-1$); see \cite[Theorem II.8.18, Remarks II.8.18.1 and III.7.9.1]{h}. Since $\PP^{r-1}(\RR)$ is Zariski dense in $\PP^{r-1}(\CC)$, there exists a hyperplane $H\in V\cap \PP^{r-1}(\RR)$ defined over $\RR$ with the properties above. To show that $Y(\CC)$ spans $H$, consider the exact sequence of sheaves
\[
0\to \Ii _{X(\CC)} \to \Ii _{X(\CC)}(1) \to \Ii_{Y(\CC),H}(1)\to 0.
\]
We have to show that $h^0(\Ii_{Y(\CC),H}(1)) = 0$. This follows from the cohomology exact sequence of the sequence above using that $h^0(\Ii _{X(\CC)}(1)) = 0$ and $h^1( \Ii _{X(\CC)}) = 0$. \\
Statement (ii) is proven as (i). 
\end{proof}

A companion rank to the admissible one is the following: 

\begin{definition}
Let $X(\CC)\subset \PP^r(\CC)$ be a projective irreducible nondegenerate variety defined over $\RR$. For any $q\in \PP^r(\RR)$, its {\it open admissible rank} is the minimal cardinality $k>0$ with the following property: for each Zariski closed proper subset $B\subsetneq X(\CC)$ there exists a set $S\subset X(\CC)$ such that $\sigma (S)=S$, $\sharp (S)=k$ and $q\in \langle S\rangle _{\RR}$.
\end{definition}

The proof of the next result is essentially the proof of \cite[Proposition 5.1]{lt}. We spell this out for sake of completeness and because a similar statement is false for the real rank (even when $X_{\reg}(\RR)\ne \emptyset$); see \cite[\S 1]{bs}.

\begin{proposition}\label{eo2}
Let $X(\CC)\subset \PP^r(\CC)$ be an integral and nondegenerate variety defined over $\RR$. Let $n = \dim X(\CC)$. Then

\begin{enumerate}

\item[(i)] $\ell _{X(\CC),\sigma}(q) \le r-n+2$ for all $q\in \PP^r(\RR)$; 

\item[(ii)]  $\ell _{X(\CC),\sigma}(q) \le r-n+1$ for all $q\in \PP^r(\RR)$ if either $r-n$ is odd or $X_{\reg}(\RR)\ne
\emptyset$;

\item[(iii)] The same conclusions hold for the {\it open admissible rank} of $q\in \PP^r(\RR)$. Fix $q\in \PP^r(\RR)$ and fix an arbitrary union $B\subset X(\CC)$ of finitely proper closed subvarieties defined over $\CC$. Then there is a set $S\subset X(\CC)\setminus B$ such that $\sharp (S) \le r-n+2$ (or $\sharp (S) \le r-n+1$), with $\sigma (S) =S$, and $q\in \langle S\rangle$. 
\end{enumerate}
\end{proposition}
\begin{proof}

(i) and (ii). Fix $q\in \PP^r(\RR)$. Let $\Gamma$ be the set of all complex hyperplanes of $\PP^r(\CC)$ containing $q$. The set $\Gamma$ is
an $(r-1)$-dimensional complex projective space defined over $\RR$, as $q\in \PP^r(\RR)$. Let $\Gamma '\subset \Gamma$
denote the closure in $\Gamma$  of the set of hyperplanes $H$ such that $H$ is tangent to $X_{\reg}(\CC)$. The set $\Gamma '$
is closed and irreducible in $\Gamma$ and $\dim \Gamma '\le r-2$. Since $q\in \PP^r(\RR)$, one has $\sigma (\Gamma ') =\Gamma
'$, i.e., $\Gamma '$ is defined over $\RR$. Since $\dim \Gamma '\le r-2$, we have $\Gamma '(\RR)\ne \Gamma(\RR)$.
Fix $H\in \Gamma (\RR)\setminus \Gamma '(\RR)$. Since $H$ and $X(\CC)$ are defined over $\RR$, the scheme $X(\CC)\cap H$ is defined over
$\RR$. Since $H\notin \Gamma '$, $X(\CC)\cap H$ is an integral and nondegenerate subvariety of $H$ (in case $n>1$) or it is a set
of $\deg (X(\CC))$ points of $H$ spanning $H$ by Lemma \ref{a3}. 

If $n=1$, we use that $\sigma (H\cap X(\CC)) = H\cap X(\CC)$ to obtain a subset $S\subseteq H\cap
X(\CC)$ with $\sigma (S)=S$, $S$ spanning $H$, with $\sharp (S) = r$ ($r$ even) or $\sharp (S) \le r+1$ ($r$ odd). If
$r$ is odd, we may take $S$ with $\sharp (S) =r$ if and only if $H\cap X(\RR) \ne \emptyset$. (The latter condition is satisfied
for some $H$ when $X_{\reg}(\RR)\ne \emptyset$.) 

Now assume $n>1$. We may use $H\cap X(\CC)$ and induct on $n$. If $X_{\reg}(\RR) \ne \emptyset$, we take a hyperplane $H$ meeting $X_{\reg}(\RR)$ transversally. 

(iii). Taking $\sigma (B)\cup B$ instead of $B$ we may assume that the closed algebraic subset $B$ of
$X(\CC)$ is defined over $\RR$ (but its irreducible components over $\CC$ themselves may not be defined over $\RR$). We may find $H\in \Gamma (\RR) \setminus \Gamma '(\RR)$ such that $\dim B\cap H \le n-2$. Thus by induction on $n$, we reduce to the case $n=1$ in which $B\cap H =\emptyset$, as desired. 
\end{proof}

\begin{proposition}\label{oo1}
Let $X(\CC)\subset \PP^3(\CC)$ be an integral and nondegenerate curve such that $X(\RR)\ne \emptyset$.
Fix $p\in X(\RR)$. Let $s$ be the number of branches of $X(\CC)$ at $p$. Then there is a set $A\subset \PP^3(\RR)$, which is a union of at most $s$ real line,
such that the following holds:

\begin{enumerate}
\item[(i)] $\ell _{X(\CC),\sigma}(q)\le 3$ for all $q\in \PP^3(\RR)\setminus A$;

\item[(ii)] For each $q\in \PP^3(\RR)\setminus A$, there is $S\subset X(\CC)$ such that $\sharp (S) \leq 3$, $\sigma (S) =S$ and $q\in \langle S\rangle _{\CC}$.
\end{enumerate}
\end{proposition}

\begin{remark}
Note that in Proposition \ref{oo1} we do not assume $X_{\reg}(\RR)\ne \emptyset$. Perhaps the most interesting situation is when $X$ is singular
and $X(\RR) \subseteq X_{\textnormal{sing}}(\CC)$. It is easy to construct such examples. For instance, starting with the conic $\Cc
=\{x^2+y^2+z^2=0\}\subset \PP^2(\CC)$ defined over $\RR$ but with $\Cc(\RR)=\emptyset$ we construct a nodal curve $X(\CC)$
with $\Cc$ as its normalization, defined over $\RR$, nodal, with $\sharp (X_{\textnormal{sing}}(\CC))=t$ and $X(\RR)
= X_{\textnormal{sing}}(\CC)$ in the following way. Fix $t$ distinct points $q_1,\dots , q_t\in X(\CC)$ such that $\sigma (q_i)\ne
q_j$ for any $i,j$. Let $X(\CC)$ be the nodal curve obtained from $\Cc$ gluing together the points $q_i$ and $\sigma (q_i)$ to
get an ordinary node. 
\end{remark}

\begin{proof}[Proof of Proposition \ref{oo1}:]
It is sufficient to prove statement (ii). For any $o\in \PP^3(\CC)$, let $\pi_o: \PP^3(\CC)\setminus \{o\}\to \PP^2(\CC)$ denote the linear projection from $o$. The morphism $\pi _o$ is defined over $\RR$ if $o\in \PP^3(\RR)$.

Fix $q\in \PP^3(\RR)$. If $q=p$ we may take $S =\{p\}$. Now assume $q\ne p$. Since $q\ne p$, the point $q'= \pi _p(q)\in \PP^2(\RR)$ is well defined. Let $X'(\CC)$ denote the closure of $\pi_p(X(\CC)\setminus \{p\})$ in $\PP^2(\CC)$. The set $X'(\CC)$ is an irreducible plane curve defined over $\RR$ and hence $\ell_{\pi _{X'(\CC),\sigma}}(q') \le 2$, by e.g. Proposition \ref{eo2}(ii). Set $B: = X'(\CC)\setminus \pi _p(X(\CC)\setminus \{p\})$. The set $B$ is finite and defined over $\RR$. Moreover, $\sharp (B)\le s$, i.e., its cardinality is at most the number $s$ of branches of $X(\CC)$ at $p$. 

Assume $q'\notin X'(\CC)$. By Proposition \ref{eo2}(iii), there is $S'\subset X'(\CC)$ such that $\sharp (S') =2$, $S'\cap B =\emptyset$, $\sigma (S')=S'$ and $q'\in \langle S'\rangle$. Thus there is $S''\subset X(\CC)\setminus \{p\}$ such that $\pi _p(S'') =S'$ and $\sigma (S'') =S''$. So $q\in \langle S'', p\rangle$, and $\sharp(S''\cup p) = 3$. If $q'\in (X'(\CC)\setminus B)\cap \PP^2(\RR)$, then there is $q''\in X(\RR)$ with $\pi_p(q'') =q'$; hence $q\in \langle p, q''\rangle$ and so $\sharp(\lbrace p, q''\rbrace) = 2$. 

In conclusion, we see that it is sufficient to take as $A$ the real part of the tangent cone to $X(\CC)$ at $p$.
\end{proof}

Utilizing the idea in \cite[Theorem 1]{bt}, we derive a Blekherman-Teitler type upper bound on the maximum admissible rank. This is the same as the one for maximum complex rank:

\begin{proposition}\label{upperbound2g}
Let $X(\CC)\subset \PP^r(\CC)$ be an integral and nondegenerate variety defined over $\RR$. Let $g= r_{\gen}(X(\CC))$ be the complex generic rank of $X(\CC)$. 
Then $\ell _{X(\CC),\sigma}(q) \le 2g$ for all $q\in \PP^r(\RR)$.
\end{proposition}

\begin{proof}
Fix $q\in \PP^r(\RR)$. By the definition of $g$, there is a non-empty Zariski open subset
$U$ of $\PP^r(\CC)$ such that $r_{X(\CC)}(p) =g$ for all $p\in U$. Set $W= U\cup \sigma (U)$. Note that $W$ is a Zariski
open subset of $\PP^r(\CC)$. We have $\sigma (W)=W$ and so $W$ is defined over $\RR$. Since the inclusion
$X(\CC)\subset \PP^r(\CC)$ is defined over $\RR$, we have $r_{X(\CC)}(z) = r_{X(\CC)}(\sigma (z))$ for all $z\in \PP^r(\CC)$. Thus
$r_{X(\CC)}(p) =g$ for all $p\in W$. Let $E(\CC)$ denote the set of all complex lines of $\PP^r(\CC)$ passing through $q$ and
intersecting $W$. This set $E(\CC)$ is a Zariski open subset of the $(r-1)$-dimensional complex projective $\PP^{r-1}(\CC)$
parametrizing all lines of $\PP^r(\CC)$ containing $q$. Since $\sigma (q) =q$ and $\sigma (W)=W$, we have $\sigma (E(\CC)) =
E(\CC)$, i.e., $E(\CC)$ is defined over $\RR$. Moreover, $E(\RR)$ is a Zariski open subset of the set $\PP^{r-1}(\RR)$ of all real
lines of $\PP^r(\RR)$ containing $q$. Since $\PP^{r-1}(\RR)$ is Zariski dense in $\PP^{r-1}(\RR)$ 
and $E(\CC)$ is dense in $\PP^{r-1}(\CC)$, we have $E(\RR)\ne \emptyset$.
Take $L\in E(\RR)$, viewed as a complex line defined over $\RR$. Since $E(\RR )\subset E(\CC)$, there is $p\in
W\cap L$. 

First assume $p\notin \PP^r(\RR)$. The line $L$ is defined over $\RR$ and so $q\in L= \langle \{p,\sigma(p) \}\rangle_{\CC}$. As $\sigma (W) = W$, we have $r_{X(\CC)}(q) = g$. If $S\subset X(\CC)$ with $\sharp (S)=g$ and $p\in \langle S\rangle _{\CC}$, then $\sigma (p)\in \langle \sigma (S)\rangle_{\CC}$ and so $q \in \langle S\cup \sigma (S)\rangle _{\CC}$. Thus $\ell_{X(\CC),\sigma}(q) \le r_{X(\CC)}(p)+r_{X(\CC)}(\sigma (p)) =2g$. 

Now assume $p\in \PP^r(\RR)$. From $p\in W$, it follows that $W\cap L$ is a non-empty Zariski open subset of $L$, i.e., $L$ minus finitely many points. Now, $W\cap L$ contains infinitely many points that are not real. Pick one of them, say $z\in L\cap W$ and so $\sigma(z)\in L\cap W$. Therefore $q\in\langle \{z,\sigma(z) \}\rangle_{\CC}=L$ and hence $\ell _{X(\CC),\sigma}(q) \le 2g$.\end{proof}

\begin{definition}
For all integers $k>0$, set
\[
W^0_k(X(\CC))= \{q\in \PP^r(\CC)\mid r_{X(\CC)}(q) = k\} \mbox{ and } W_k(X(\CC))= \overline{W^0_k(X(\CC))}. 
\]
If $\sigma _{k}(X(\CC))\subsetneq \PP^r(\CC)$, then $W_k(X(\CC)) = \sigma_k(X(\CC))$. 

Similarly, let $W^0_{k,\sigma}(X(\CC))$ be the set of all $q\in \PP^r(\RR)$ such that $\ell _{X(\CC),\sigma}(q) =k$.
For any label $\mu =(a,b)$ of weight $k$, let $W^0_{k,\sigma}(X(\CC),\sigma)_{\mu}$ be the set of all $q\in W^0_{k,\sigma}(X(\CC))_{\mu}$
with $\mu$ as one of its labels. Therefore: 
\[
W^0_{k,\sigma}(X(\CC))= \bigcup_{\mu} W^0_{k,\sigma}(X(\CC))_{\mu}
\]
We may have $W^0_{k,\sigma}(X(\CC))_{\mu}\cap W^0_{k,\sigma}(X(\CC))_{\mu'} \ne \emptyset$ for two
different labels $\mu$ and $\mu '$; see Example \ref{twolabels}. Their closures in Euclidean topology are $W_{k,\sigma}(X(\CC))_{\mu}$ and $W_{k,\sigma}(X(\CC))$ respectively. 
\end{definition}

\begin{remark}
Let $X(\CC)\subset \PP^r(\CC)$ be an integral and nondegenerate variety. 
Since images of semialgebraic sets by algebraic maps are semialgebraic \cite[Proposition 2.2.7]{bcr}, each set
$W^0_{k,\sigma}(X(\CC))_{\mu}$ is semialgebraic. Hence each $W^0_{k,\sigma}(X(\CC))$ is semialgebraic \cite[Proposition
2.2.2]{bcr}. Their closures in Euclidean topology, $W_{k,\sigma}(X(\CC))_{\mu}$ and $W_{k,\sigma}(X(\CC))$ respectively, are semialgebraic as
well \cite[Proposition 2.2.2]{bcr}. 
\end{remark}

\begin{remark}\label{eu0}
Let $E$ be a non-empty real semialgebraic set. Then there is an integer $k>0$ (not uniquely determined)
such that $E =\cup _{i=1}^{k} E_i$ with $E_i$ semialgebraic, $E_i\cap E_j=\emptyset$ for all $i\ne j$, where
each $E_i$ is homeomorphic to an open hypercube $]0,1[^{d_i}$ with $d_i\in \NN$ \cite[Theorem 2.3.6]{bcr}. The dimension of $E$, denoted 
$\dim E$, is the maximum of the integers $d_i$'s. This is the Hausdorff dimension of $E$, i.e., its dimension as locally compact topological space. 
One has $\dim E = \dim \overline{E}$, where $\overline{E}$ is its closure in Euclidean topology. 
\end{remark}

\begin{remark}\label{eur1}
Let $W$ be a real algebraic manifold with pure dimension; typically, $W=\PP^r(\RR)$. Let $A, B\subseteq W$ be semialgebraic sets. 
Let $A^{\circ}$ and $\overline{A}$ denote interior and closure of $A$ respectively. With these notations, we have: 
\begin{itemize}
\item $(\overline{A})^{\circ}\subseteq \overline{A^{\circ}}$ \cite[Lemma 2.1]{bbo};
\vspace{2mm}
\item $\overline{A\cup B} = \overline{A}\cup\overline{B}$; 
\vspace{2mm}
\item $(A\cup B)^{\circ} = A^{\circ}\cup B^{\circ}$. 
\end{itemize}
\end{remark}

\begin{definition}\label{rho}
Let $\KK$ be an arbitrary field. A zero-dimensional scheme $Z\subset \PP^r(\KK)$ is said to be linearly independent if $\dim \langle Z\rangle _{\KK} = \deg(Z) - 1$.
Let $X\subset \PP^r(\KK)$ be an integral and nondegenerate variety defined over $\KK$.  Let $\rho(X)$ (resp. $\rho'(X)$) denote the maximal integer $t>0$ such that each zero-dimensional scheme (resp. each {\it reduced} zero-dimensional scheme) $Z\subset X(\KK)$  with $\deg(Z) = t$ is linearly independent in $\PP^r(\KK)$. It is clear that $\rho(X)\leq \rho'(X)$. 
\end{definition}

\begin{remark}\label{e0}
Assume $\rho (X(\CC)) \ge 2k$ (resp. $\rho (X(\CC))'\ge 2k$). For each $q\in \PP^r(\CC)$ there is at most one zero-dimensional
scheme (resp. finite set) $Z\subset X(\CC)$ such that $\deg (Z)\le k$ and $q\in \langle Z\rangle$. Indeed, on the contrary, suppose 
$q\in \langle Z_1\rangle \cap \langle Z_2 \rangle$, where $Z_1\neq Z_2$; however $\deg(Z_1\cup Z_2) \leq 2k$ and thus $Z_1\cup Z_2$ is linearly independent. 
This is in contradiction with $\langle Z_1\rangle \cap \langle Z_2 \rangle\neq \emptyset$. 

Note that, under these assumptions, $\sigma_k(X(\CC))$ is not defective and so its dimension is $k(\dim X(\CC)+1)-1$.
\end{remark}

\begin{definition}\label{labelcactus}
For any $q\in \PP^r(\CC)$, the {\it cactus rank} $cr_{X(\CC)}(q)$ of $q$ is the minimal degree of a zero-dimensional scheme $Z\subset X(\CC)$ such that $q\in \langle Z\rangle$.
For any $q\in \PP^r(\RR)$ the {\it admissible cactus rank} $cr_{X(\CC),\sigma}(q)$ of $q$ is the minimal degree of a zero-dimensional scheme $Z\subset X(\CC)$ such that $\sigma (Z)=Z$ and $q\in \langle Z\rangle$.
\end{definition}

Another implication of the assumption in Remark \ref{e0} that will be employed is as follows. Recall that given $q\in \PP^r(\CC)$, its {\it border rank} $b_{X(\CC)}(q)$ is
the minimal integer $k$ such that $q\in \sigma_k(X(\CC))$. 

\begin{proposition}\label{crvsbr}
Let $\rho(X)\geq k$. Then, for every $q\in \sigma_k(X(\CC))$, one has
\[
cr_{X(\CC)}(q)\leq b_{X(\CC)}(q). 
\]
\begin{proof}
Let $q\in \sigma_k(X(\CC))$ be such that $b_{X(\CC)}(q)=k$. By definition, $q$ is a limit of an (algebraic) curve of points $p_{\epsilon}$ with $r_{X(\CC)}(p_{\epsilon}) = k$. 
For each $\epsilon$, there exists a finite set $S_{\epsilon}\subset X(\CC)$ of $k$ points such that $p_{\epsilon}\in \langle S_{\epsilon}\rangle$. 
Since the Hilbert scheme $\textnormal{Hilb}_k(X(\CC))$ of schemes of degree $k$ on $X(\CC)$ is projective, there exists a flat limit $S_0$ of the family
$S_{\epsilon}$. Since $\rho(X)\geq k$, all the $S_{\epsilon}$ and their limit $S_0$ are linearly independent. Therefore the limit
of the spans $\langle S_{\epsilon}\rangle$ coincides with the span $\langle S_0\rangle$ (this is not always the case, only one inclusion
holds in general). Hence $q\in \langle S_0\rangle$. 
Thus $cr_{X(\CC)}(q)\leq k = b_{X(\CC)}(q)$.
\end{proof}

\end{proposition}

Now we introduce the notion of a label for a zero-dimensional scheme $Z\subset X(\CC)$ defined over $\RR$: 

\begin{definition}\label{labelforscheme}
Let $Z\subset X(\CC)$ be a zero-dimensional scheme. Let $Z_1,\dots ,Z_s$ be its connected components. Set $\{p_i\}= (Z_i)_{\red}$, $1\le i \le
s$, $d= \deg (Z)$ and $d_i= \deg (Z_i)$. We have $d=d_1+\cdots +d_s$ and $S= Z_{\red} =\{p_1,\dots ,p_s\}$.
By definition, $Z$ is defined over $\RR$ if and only if $\sigma (Z)=Z$ ($\sigma$ induces a permutation on connected components of $Z$). We have $\sigma (S)=S$ and hence $S$ has a label $(a,b)$. The {\it scheme-label} of $Z$ is the tuple $(a,b; d_1,\ldots,d_s)$, where the right-hand side records the degrees of its 
connected components, $b$ is the sum of the degrees of the $Z_i$ such that $\sigma (Z_i) =Z_i$ and $2a$ is the sum of the other $d_i$'s. \end{definition}

\begin{remark}
A zero-dimensional scheme $Z$ defined over $\RR$ is connected if and only if it has label $(0,1; d, 0, \ldots, 0)$. Easy examples show that some $Z$ may have subschemes not defined over $\RR$.  
\end{remark}

A class of connected zero-dimensional schemes defined over $\RR$ such that all of their subschemes
are defined over $\RR$ is as follows. Let $Z\subset X(\CC)$ be a connected zero-dimensional scheme. We say that $Z$ is \emph{curvilinear} if its Zariski tangent space has dimension at most $1$. If $Z\subset X_{\reg}(\CC)$ the scheme $Z$ is curvilinear if and only if it is contained in a smooth curve. Since $Z\subset \PP^r(\CC)$, a curvilinear scheme is always contained in a smooth curve contained in $\PP^r(\CC)$. This implies the fact that if $Z$ is connected and curvilinear, then for every integer $t$ such that $1\le t\le \deg (Z)$ there is a unique scheme $Z_t\subset Z$. This uniqueness shows that, if $\sigma (Z) =Z$, then $\sigma (Z_t)=Z_t$. Hence every subscheme of a curvilinear connected zero-dimensional scheme defined over $\RR$ is defined over $\RR$.
A zero-dimensional scheme $Z$ whose connected components are curvilinear is also called curvilinear.

\begin{proposition}\label{e1}
Let $X(\CC)$ be as above with $n=\dim _{\CC} X(\CC)$. Assume $2k\le \rho' (X(\CC))$.  
\begin{enumerate}
\item[(i)] Let $q\in \sigma _k^0(X(\CC))\cap \PP^r(\RR)$. Then $\ell _{X(\CC),\sigma}(q) =k$ and $q$ has a unique label.

\item[(ii)] Assume $X_{\reg}(\RR) \ne \emptyset$ and $2k \le \rho (X(\CC))$. Let $(a,b;d_1,\ldots,d_s)$ be a scheme-label of weight $k$. This label is a scheme-label of some 
$q\in \left(\sigma _k(X(\CC))\setminus \sigma _{k-1}(X(\CC))\right)\cap \PP^r(\RR)$ and the label may be realized by a curvilinear scheme. Such a point $q$
satisfies $cr _{X(\CC)}(q) =k$. 
\end{enumerate}
\end{proposition}

\begin{proof}
(i). By Remark \ref{e0}, the set $\Ss (X(\CC),q)$ is a singleton:
\[
\Ss (X(\CC),q) = \{S\}.
\]
As $\sigma (q) = q$, then $\sigma (S) =S$. So $\ell _{X(\CC),\sigma}(q) =k$ and $q$ has a unique label, that of $S$.

(ii). Take $d_1,\ldots ,d_e$ such that $d_i>0$ for all $i$, $\sum _{i=1}^{e} d_i =a$ and $\sum _{i=2e+1}^{2e+f} d_i =b$. Write $d_{i+e}= d_i$ for all $i=1,\dots ,e$. 

Fix $p_1,\dots ,p_e\in X_{\reg}(\CC)\setminus X_{\reg}(\RR)$ such that $\sigma (p_i)\ne p_j$ for all $i,j$ (for instance take general $p_1,\dots ,p_e$). Fix $p_{2e+1},\dots ,p_{2e+f}\in X_{\reg}(\CC)$. For any $i=1,\dots ,e$ take any connected and curvilinear scheme
$Z_i \subset X(\CC)$ such that $\deg (Z_i) =d_i$ and $(Z_i)_{\red} =\{p_i\}$. For $1\le i \le e$, set $Z_{i+e} = \sigma (Z_i)$. For $2e+1 \le i \le 2e+f$, let $Z_i\subset X_{\reg}(\RR)$ denote a connected zero-dimensional scheme defined over $\RR$, with $\deg (Z_i) =d_i$ and $(Z_i)_{\red} =\{p_i\}$. Set $Z= Z_1\cup \cdots \cup Z_{2e+1}$. Such a $Z$ is a curvilinear scheme with scheme-label $(a,b;d_1,\ldots,d_s)$. Take any $q\in \langle Z \rangle$ such that $q\notin \langle Z'\rangle $ for any $Z'\subsetneq Z$; we may take as $q$ a general element of $\langle Z\rangle$, because $Z$ is linearly independent (since $k\le \rho (X(\CC))$) and $Z$ has only finitely many subschemes (because $Z$ is curvilinear). By Remark \ref{e0} and Proposition \ref{crvsbr}, $cr _{X(\CC)}(q) =k$ and $Z$ is the only scheme evincing the cactus rank of $q$. \end{proof}

\section{Typical labels}\label{Seu}

The definition of typical rank is a key notion for real ranks of (real) algebraic varieties. Recently this topic has witnessed a tremendous amount of results; see for instance \cite{abc, bb, b, bbo, BrSt, mm,mmsv, v}. We extend some of them here to the setting of {\it typical labels} which will be defined in a moment. 

First, let $X(\CC) \subset \PP^r(\CC)$ be a nondegenerate projective variety defined over $\RR$. The \emph{generic rank} $r_{\gen}(X(\CC))$ of $X(\CC)$ is the minimal integer $k$ such that $\sigma _k(X(\CC)) = \PP^r(\CC)$. The integer $r_{\gen}(X(\CC))$ is the unique integer such that there is a
non-empty Zariski open subset $U \subset \PP^r(\CC)$ with $r_{X(\CC)}(q)= r_{\gen}(X(\CC))$ for all $q\in U$. A different scenario arises if one considers 
open subsets for the Euclidean topology in $\PP^r(\RR)$ instead. An integer $k$ is a \emph{typical rank} of $X(\CC)$
if there exists a non-empty Euclidean open subset $U_k\subset \PP^r(\RR)$ such that $r_{X(\RR)}(q) =k$ for all $q\in U_k$, and there might be several of them. (For this definition, as far as we know, most of the references above either assume $X_{\reg}(\RR)\ne\emptyset$ or study examples where this condition is satisfied.) To define typical labels and their weights we first state the general definition and then shift gears to examples with $X_{\reg}(\RR)\ne \emptyset$.

\begin{definition}\label{typicallab}
A label $(a,b)$ is a {\it typical label} of $X(\CC)$ if there is a non-empty Euclidean open subset $U\subset \PP^r(\RR)$ such that
all $q\in U$ have $(a,b)$ as one of their labels with weight $2a+b = \ell_{X(\CC),\sigma}(q)$. The {\it typical weight-labels} of $X(\CC)$
are the weights of the typical labels.
\end{definition}

The main results in \cite{bal} essentially establish that in many interesting cases there are typical labels with at most {\it two} consecutive weights and that in a few, very particular, cases all typical labels have the same weight, see, e.g., our Theorem \ref{o1}.

\begin{proposition}\label{typicallabels}
Let $X(\CC)\subset \PP^r(\CC)$ be an integral and nondegenerate complex projective variety with $X_{\reg}(\RR)\ne \emptyset$. Let $g=r_{\gen}(X(\CC))$. Then all labels $(a,b)$ with $2a+b =g$ are typical.
\end{proposition}

\begin{proof}
Fix $(a,b)\in \NN^2$ such that $2a+b =g$. We need to show the existence of a non-empty Euclidean subset $U$ of  $\PP^r(\RR)$ such
that for each $q\in U$ there is $S\subset X(\CC)$ with label $(a,b)$ and $q\in \langle S\rangle _{\RR}$.
Fix a general (in the Zariski topology) subset $A\subset X(\CC)$, such that $A\cap X(\RR ) =\emptyset$ and $\sigma (A)\cap A=\emptyset$.
Fix $B\subset X(\RR)$, which is general in $X(\RR)$ in the Zariski topology of $X(\RR)$. Since $X_{\reg}(\RR)\ne \emptyset$, we may take $B$ inside the smooth locus. Set $S= A\cup \sigma (A) \cup B$. By construction, $S$ has label $(a,b)$. The set of all such $S$ is Zariski dense in the set of all subsets of $X(\CC)$ with cardinality $g$. Varying $A$ and $B$ as above defines a semialgebraic set $E$ of points $q$ with $r_{X(\CC)}(q) = g$ and admitting a label $(a,b)$. The closure
of $E$ is $\PP^r(\RR)$. Therefore the semialgebraic set $E$ contains a non-empty Euclidean open subset. 
\end{proof}

\begin{remark}\label{balresults}
We now record some of the statements proven in \cite{bal}:
\begin{enumerate}
\item[(i)] Let $X(\CC)\subset \PP^{n+1}(\CC)$ be an integral hypersurface defined over $\RR$. Each $q\in \PP^{n+1}(\RR)\setminus
X(\RR)$ has a label of weight $2$; see \cite[Proposition 1.2]{bal}.

\item[(ii)] Let $X(\CC) \subset \PP^r(\CC)$ be an integral and nondegenerate variety defined over $\RR$ such that $X_{\reg}(\RR)\ne \emptyset$. Assume generic identifiability holds for $\sigma_{g-1}(X(\CC))$, where $g$ is the generic complex rank, the minimal integer such that $\sigma _g(X(\CC)) =\PP^r(\CC)$. Then $g$ is a typical weight-label and no integer $\ge g+2$ is a typical weight-label; see \cite[Theorem 1.4]{bal}.

\item[(iii)] Fix integers $n\ge 1$ and $d\ge 3$. Assume $(n,d) \notin \{(2,6), (3,4), (5,3)\}$. Set $r= \binom{n+d}{n}-1$ and $g= \lceil (r+1)/(n+1)\rceil$. Let $X(\CC)\subset \PP^r(\CC)$ be the $d$-th Veronese embedding. Then $g$ is a typical weight-label and no integer $\ge g+2$ is a typical weight-label; see \cite[Theorem 1.5]{bal}.

\item[(iv)] Let $X(\CC )\subset \PP^r(\CC)$ be an integral and nondegenerate curve defined over $\RR$ and with $X(\RR)$ infinite.
Then $\lfloor (r+2)/2\rfloor$ is a typical rank and no integer $\ge 2+ \lfloor (r+2)/2\rfloor$ is a typical weight-label; see \cite[Proposition 1.6]{bal}.
\end{enumerate}
\end{remark}

\begin{theorem}\label{v1}
Fix an odd integer $r\ge 3$. Set $g= (r+1)/2$. Let $X(\CC)\subset \PP^r(\CC)$ be a linearly normal elliptic curve defined
over $\RR$ with $X(\RR )\ne \emptyset$.
\begin{enumerate}
\item[(i)] The integers $g$ and $g+1$ are typical weight-labels of $X(\CC)$,  and all
pairs $(a,b)\in \NN^2$ with $2a+b = g$ are typical labels of $X(\CC)$.

\item[(ii)] The weight $g+1$ is a typical weight-label and it is the maximum of such. Therefore $g$ and $g+1$ are the only typical weight-labels.
\end{enumerate}
\end{theorem}

\begin{proof}
(i). By \cite[Proposition 5.8]{cc1}, there is a non-empty Zariski open subset $V\subset \PP^r(\CC)$ such that $r_{X(\CC)}(q)=g$ and $\sharp (\Ss (X(\CC),q)) =2$,
for all $q\in V$, with $\langle S_1\rangle_{\CC} \cap \langle S_2\rangle _{\CC} =\{q\}$, where $\{S_1,S_2\} =\Ss (X(\CC),q)$. Since $\PP^r(\RR)$ is Zariski dense, there exists  $q\in V\cap \PP^r(\RR)$. Write $\Ss (X(\CC),q) =\{S_1,S_2\}$. Therefore two cases may occur:
\begin{itemize}
\item[({\bf a})] $\sigma (S_1) =S_1$;
\item[({\bf b})] $\sigma (S_1)\ne S_1$.
\end{itemize}

\quad ({\bf a}) Assume $\sigma (S_1) =S_1$. Since $\sigma (q) =q$, we have $\sigma (S_2)=S_2$. Fix a label $(a,b)$ with weight
$2a+b=g$. Fix $a$ general points $u_1,\dots ,u_a$ of $X(\CC)$ (i.e., a general
subset of $X(\CC)$ for the Zariski topology of cardinality $a$) and $b$ general points $v_1,\dots ,v_b$ of $X(\RR)$ (i.e., a general $b$-tuple $(v_1,\dots ,v_b)\in X(\RR)^b$, where $X(\RR)$ is either a circle or a disjoint union of two circles). We may choose the $a$ points so that $\sigma
(u_i)\ne u_j$ for any $i, j$. Set
$S=\{v_1,\dots ,v_b,u_1,\dots ,u_a,\sigma _1(u_1),\dots ,\sigma (u_a)\}$. If $q\in \langle S\rangle _{\RR}$ and $q\notin \langle
S'\rangle _{\CC}$ for any $S'\subsetneq S$, then $S\in \Ss(X,q)$. Thus each label $(a,b)$
with weight $2a+b=g$ is typical and the non-empty subset of $V$ coming from case ({\bf a}) has at least one of these $(a,b)$ as its label. 

\quad ({\bf b}) Assume $\sigma (S_1)\ne S_1$. Since $\sigma (q)=q$ and $\Ss (X,q) =\{S_1,S_2\}$, we derive
$\sigma (S_1) =S_2$ and $\sigma (S_2)=S_1$. Therefore $q$ has no label of weight $g$. 

Since a general $S\subset X(\CC)$ with 
$\sharp (S) =g$ is not fixed by the conjugation $\sigma$, one has $S\cap \sigma (S)=\emptyset$. The intersection $\langle S\rangle _{\CC}\cap \langle \sigma(S)\rangle _{\CC}$ is a single point $q'$. The uniqueness of $q'$ gives $\sigma (q') = q'$. Thus we see that case ({\bf b}) occurs in a non-empty Euclidean open
subset of $\PP^r(\RR)$. Hence there are typical weights $>g$. 

(ii). By Remark \ref{balresults}(iv), there is no typical weight $\ge g+2$. 
\end{proof}

\begin{theorem}\label{w11}
Let $X(\CC)\subset \PP^3(\CC)$ be an integral and nondegenerate curve defined over $\RR$ with only planar singularities. Let
$d = \deg (X(\CC))$ and $p_a=p_a(X(\CC))$ be its arithmetic genus. Assume $(d-1)(d-2)/2 - p_a \equiv 1 \pmod{2}$.
Then: 
\begin{enumerate}
\item[(i)] $(1,0)$ is a typical label;
\item[(ii)] $(0,2)$ is a typical label if and only if $X_{\reg}(\RR)\ne \emptyset$, i.e., if and only if $X(\RR)$ is infinite;
\item[(iii)] No label of weight $>2$ is typical.
\end{enumerate}
\end{theorem}
\begin{proof}
(i) and (ii) are by Proposition \ref{typicallabels}. \\
(iii). Let $\tau (X(\CC))$ denote the tangential variety of $X(\CC)$, i.e., the closure of the union of all tangent lines to
$X_{\reg}(\CC)$. It is defined over $\RR$ and hence its intersection with $\PP^3(\RR)$ has real dimension $\le 2$. The curve $X(\CC)$ has
only finitely many singular points, say $z_1,\dots ,z_k$ with $k\ge 0$. 
By assumption, each Zariski tangent space $T_{z_i}(X(\CC))$ is a complex plane; this is real if and only if $z_i\in
X(\RR)$. If $z_i\notin \PP^3(\RR)$, then $k\ge 2$ and there is a unique $j\in
\{1,\dots ,k\}$ with $\sigma (z_i) =z_j$. We have $\sigma (T_{z_i}(X(\CC))) = T_{z_j}(X(\CC))$ and $T_{z_i}X(\CC)\cap
T_{z_j}X(\CC)$ is defined over $\RR$. Let $\Delta \subset \PP^3(\CC)$ denote the closure of the union of all lines $L\subset \PP^3(\CC)$
such that $\sharp (L\cap X(\CC)) \ge 3$. The set $\Delta$ is a closed algebraic subset of $\PP^3(\CC)$ define over $\RR$: indeed, since $X(\CC)$ is defined over $\RR$, for every line $L\subset \Delta$, one has $\sigma(L)\subset \Delta$. The Trisecant lemma (\cite[p. 109]{acgh}) shows that $\dim _{\CC} \Delta \le 2$ and hence $\Delta \cap \PP^3(\RR)$ has real dimension $\le 2$. Moreover, since $X(\CC)$ is a curve, there are only finitely many points $q\in \PP^3\setminus X(\CC)$ such that $\pi _{q|X(\CC)}$ is {\it not} birational onto its image; the set $\mathfrak S$ of such points is called the {\it Segre set} of $X(\CC)$ in \cite{cc}. 

For each $i\in \{1,\dots ,k\}$, let $\Delta _i$ denote the join of $z_i$ and $X(\CC)$, i.e., the closure in $\PP^3(\CC)$ of the union of all lines spanned by $z_i$ and a point of $X(\CC)\setminus \{z_i\}$. If $\sigma (z_i) =z_i$ we have $\sigma (\Delta _i) =\Delta _i$, whereas if $\sigma (z_i) =z_j$ with $j\ne i$ we have $\sigma (\Delta _i)=\Delta _j$. Thus $\Delta _1\cup \cdots \cup \Delta _k$
is defined over $\RR$ and either it is empty (case $k=0$) or it is a surface (case $k>0$). Thus $\Delta _1\cup \cdots \cup \Delta _k  \cap \PP^3(\RR)$ has real dimension $\le 2$.

Thus there is an Euclidean open subset $U\subset \PP^3(\RR)$ such that $q\notin (\tau (X(\CC)) \cup \Delta
\cup T_{z_1}X(\CC)\cup \cdots \cup T_{z_k}X(\CC)\cup \Delta _1\cup \cdots \Delta_k \cup \mathfrak S)$ for all $q\in U$. Fix $q\in U$ and call $\pi_q:
\PP^3(\CC)\setminus \{q\}\to \PP^2(\CC)$ the projection away from $q$. Since $q\in \PP^3(\RR)$, $\pi _q$ is defined over $\RR$. Since $q\notin \tau
(X(\CC))$, we have $q\notin X(\CC)$ and so $f= \pi _{q|X(\CC)}: X(\CC) \to \PP^2(\CC)$ is a morphism defined over $\RR$. Let $Y(\CC) = f(X(\CC))$. 

Since $q\notin \mathfrak S$, $\deg (f)=1$ and so $Y(\CC)$ has degree $d$; moreover, since $q\notin \Delta$, every fiber of $f$ contains at most two points of $X(\CC)$. Since $q\notin T_{z_1}X(\CC)\cup \cdots \cup T_{z_k}X(\CC)\cup \Delta _1\cup \cdots \cup \Delta _k$, $f$ maps isomorphically
each singular point $z_i$ of $X$ to some branch of $Y(\CC)$ at $f(z_i)$ and there is no $o\in X(\CC) \setminus \{z_i\}$ with $f(o) =f(z_i)$. Since $q\notin \tau (X(\CC))$, the differential of $f$ is injective at each smooth point of $X(\CC)$. The degree $d$ plane curve $Y(\CC)$ has arithmetic genus $(d-1)(d-2)/2$.
Thus $Y(\CC)$ has (besides the $k$ singular points $f(z_i)$, $1\le i \le k$) exactly $(d-1)(d-2)/2 - p_a$ ordinary nodes. This set $S$ of extra nodes satisfies $\sigma (S) =S$. Hence if $\sharp (S)$ is odd, we have $S\cap \PP^2(\RR) \ne \emptyset$. Take any $u\in S\cap \PP^2(\RR)$. The latter corresponds to $u_1, u_2\in X(\CC)$ such that $u_1\ne u_2$. Since $\sigma (u)=u$, either $\sigma (u_1)=u_1$ and $\sigma (u_2)=u_2$ (and hence $(0,2)$ is a label of $q$ with minimal weight) or $\sigma (u_1) =u_2$ and $\sigma (u_2) =u_1$. Since $\PP^3(\RR)\setminus U$ has real dimension at most $2$, no other label is typical.
\end{proof}

\begin{definition}
Let $X(\CC)$ be as above defined over $\RR$ and let $a\ge 0$ be a nonnegative integer. 
For any $b\in \NN$, let $A_{a;b}$ be the set of all $q\in \PP^r(\RR)$ such that $b$ is the minimal integer such that $(a,b)$ is one of the labels of $q$. This means
that there exists $S$ with label $(a,b)$ such that $q\in \langle S\rangle_{\RR}$, but there is {\it no} $S'$ with label $(a,b-1)$ such that $q\in \langle S\rangle_{\RR}$. (The latter condition is considered always satisfied if $b=0$.) The set $S$ with label $(a,b)$ might be a {\it non-minimal} admissible decomposition for $q$. 
\end{definition}

Recall the notation: $A^{\circ}_{a;b}$ denotes the interior of $A_{a;b}$. 
\begin{definition}
Fix $a\in \NN$. An integer $b\ge 0$ is called {\it $a$-typical} with respect to $X(\CC)$ if $A_{a;b}\subset \PP^r(\RR)$ contains a non-empty
Euclidean open subset.
\end{definition}

Now, we follow the proof of \cite[Theorem 2.2]{bbo}. Let $b_a$ denote the minimal $b\in
\NN$ such that $(a,b)$ is $a$-typical. For the next two propositions, we assume $X(\RR)_{\reg}\neq \emptyset$.

\begin{proposition}\label{eu5}
Assume $A_{a,b}^{\circ} \ne \emptyset$ and $(a,b+1)$ not $a$-typical. Then $\overline{\cup _{b_a\le c\le b} A_{a;c}^{\circ}} =
\PP^r(\RR)$.
\end{proposition}

\begin{proof}
Set $A = \cup _{b_a\le c\le b} \overline{A_{a;c}^{\circ}}$. By Remark \ref{eur1},  $A$ is the closure of  the interior of $\cup _{b_a \le c\le b} A_{a;c}$ in $\PP^r(\RR)$. Since $b+1$ is not $a$-typical and by \cite[Lemma 2.1]{bbo}, we have $A^{\circ}_{a,b+1}\subseteq A$ and hence $\cup _{c\le b+1} A^{\circ}_{a;c} \subseteq A$. Fix $p\in A$. By definition of closure, there is a sequence $\{p_i\}_{i\in \NN}$ converging to $p$ with $p_i \in (\cup
_{b_a\le c\le b} A_{a;c})^{\circ} = \cup _{b_a\le c\le b} A^{\circ}_{a;c}$. Possibly passing to a subsequence, we may assume the existence of $c$
such that $b_a\le c\le b$ and $p_i\in A_{a;c}^{\circ}$ for all $i$. Hence each element of $J^0_{\RR}(\{p_i\},X(\RR))$ has $(a,c+1)$ as one of its
labels and $J^0_{\RR}(\{p_i\},X(\RR)) \subseteq A$. Since $A$ is closed, we obtain $J_{\RR}(\{p\},X(\RR))\subseteq A$.
By induction, we  derive that $A$ contains the join between $p$ and an arbitrary number of joins of $X(\RR)$. Since $X(\RR)$ spans
$\PP^r(\RR)$, we have $A = \PP^r(\RR)$. 
\end{proof}

Proposition \ref{eu5} shows a similarity between typical ranks and typical labels; see \cite[Corollary 2.3]{bbo} for typical ranks.

\begin{corollary}\label{eu6}
Assume that $(a,b)$ and $(a,c)$, for some $c\ge b+2$, are $a$-typical labels. Then all $(a,y)$, $b<y<c$, are $a$-typical.
\end{corollary}

Fix $b\in \NN$. We say that a label $(a,b)$ is {\it typical-$b$} if there is a non-empty open subset $U$ of $\PP^r(\RR)$ such that each $q\in U$ has $(a,b)$ as one of its label and $(a-1,b)$ is not one of its labels. With these definitions, a similar proof to the one of Proposition \ref{eu5} yields the next result and its corollary:  

\begin{proposition}\label{eu7}
Assume $A_{a,b}^{\circ} \ne \emptyset$ and $(a+1,b)$ is typical-$b$. Then $\overline{\cup _{b_a\le x\le b} A_{a,b}^{\circ}} =
\PP^r(\RR)$.
\end{proposition}

\begin{corollary}\label{eu8}
Assume that $(a,b)$ and $(d,b)$, $d\ge a+2$ are typical-$b$ labels. Then all $(y,b)$, $a<y<d$, are typical-$b$. 
\end{corollary}

\section{Rational normal curves}

In this section, we show our main result for rational normal curves: 

\begin{theorem}\label{o1}
Fix an integer $r\ge 2$. Let $X(\CC) \subset \PP^r(\CC)$ denote the degree $r$ rational normal curve with $X(\RR)\simeq
\PP^1(\RR)$. For any $q\in \PP^r(\RR)$, we have
\[
\ell _{X(\CC),\sigma}(q) = r_{X(\CC)}(q) \mbox{ and }  cr_{X(\CC),\sigma}(q) = cr_{X(\CC)}(q).
\]
\end{theorem}

\begin{proof}
It is clear from the definitions that $\ell _{X(\CC),\sigma}(q) \ge r_{X(\CC)}(q)$ and $cr_{X(\CC),\sigma}(q) \ge
cr_{X(\CC)}(q)$. Since the complex variety $X(\CC)$ is the rational normal curve, for any $q\in \PP^r(\CC)$, its cactus rank coincides with its border rank,
$b_{X(\CC)}(q) \le \lfloor (r+2)/2\rfloor$, and either $b_{X(\CC)}(q) = r_{X(\CC)}(q)$ or $r_{X(\CC)}(q) +b_{X(\CC)}(q) =r+2$ ({\it Sylvester's theorem}, see e.g. \cite{CoSe}).

Now assume $q\in \PP^r(\RR)$. For any $t\in \NN$, let $\mathrm{Div}^t(X(\CC))$ denote the set of all
degree $t$ effective divisors of $X(\CC)$. Since $X(\CC)\cong \PP^1(\CC)$, it is well-known that the algebraic variety
$\mathrm{Div}^t(X(\CC))$ is isomorphic to $\PP^t(\CC)$. Since the isomorphism between $X(\CC)$ and $\PP^1(\CC)$ is defined
over $\RR$, $\mathrm{Div}^t(X(\CC))$ is defined over $\RR$ and $\mathrm{Div}^t(X(\CC))(\RR) \cong \PP^t(\RR)$. 
We divide the rest of the proof into cases, according to the value of $b_{X(\CC)}(q)$, the border rank of $q$: \\
\indent \quad ({\bf a}) First assume $b_{X(\CC)}(q) \le \lfloor (r+1)/2\rfloor$. Since $\rho (X(\CC)) =r+1\ge 2b_{X(\CC)}(q)$, there is a
unique zero-dimensional scheme
$Z\subset X(\CC)$ such that $\deg (Z) = b_{X(\CC)}(q)$ and $q\in \langle Z\rangle$.
Since the embedding is defined over $\RR$, $\sigma (Z)$ is the unique scheme evincing the border or cactus rank of $\sigma (q)$.
Since $q\in \PP^r(\RR)$, we have $\sigma (q)=q$ and hence $\sigma (Z)=Z$. Thus $Z$ has a label. Hence 
$cr_{X(\CC),\sigma}(q)\le b_{X(\CC)}(q)$. Thus $cr_{X(\CC),\sigma}(q) = cr_{X(\CC)}(q) = b_{X(\CC)}(q)$.\\
\indent \quad ({\bf a1}) Assume $r_{X(\CC)}(q) = b_{X(\CC)}(q)$. Again, the uniqueness of $Z$ implies that $\sigma (Z)=Z$, i.e., $Z$ has a
label. Since $Z$ has a label, $\ell _{X(\CC), \sigma}(q) \le b_{X(\CC)}(q)$.\\
\indent \quad ({\bf a2}) Assume $r_{X(\CC)}(q)\ne b_{X(\CC)}(q)$. Thus $r_{X(\CC)}(q) = r+2-b_{X(\CC)}(q)$. Since $\rho (X(\CC)) = r+1$, for each $A\in \mathrm{Div}^{s}(X(\CC))$ we have $\dim \langle A\rangle = \min \{r,s-1\}$. Take any $S\subset X(\CC)$ such that $\sharp (S) =r+2-b_{X(\CC)}(q)$ and $q\in \langle S\rangle _{\CC}$. Since $q\in \langle Z\rangle \cap \langle S\rangle$, we get $h^1(X(\CC),\Ii _{Z\cup S}({r})) >0$. Since $\rho (X(\CC)) =r+1$, one has $\deg (Z\cup S) \ge r+2$. Thus $Z\cap S=\emptyset$. Let $U(q)$ denote the set of all divisors $S\in \mathrm{Div}^{r+2-b_{X(\CC)}(q)}(X(\CC))$ such that $\langle Z\rangle \cap \langle S\rangle =\{q\}$. We have just seen that the set $U'(q)$ of all reduced $A\in U(q)$ is the set of decompositions of $q$. By Sylvester's theorem, $U'(q)$ is a non-empty Zariski open subset of  a projective space of positive dimension and $U(q)$ is its closure in $\mathrm{Div}^{r+2-b_{X(\CC)}(q)}(X(\CC))$. 

Fix any $B\in U(q)$. Since $\sigma (q)=q$, we have $\sigma (B)\in U(q)$ for all $B\in U(q)$. Thus $U(q)$ is defined over $\RR$. Sylvester's theorem provides the existence of $S\in U(q)$ formed by $r+2-b_{X(\CC)}(q)$ distinct points, a rank decomposition of $q$. Thus a general $B\in U(q)$ is reduced. Since $U(q)(\RR)$ is Zariski dense in $U(q)$, we get the existence of a reduced $S\in U(q)(\RR)$. Since $S\in U(q)(\RR)$, we have $\sigma (S) =S$, i.e., $S$ has a label. Since $S\in U(q)$, we get $\ell _{X(\CC),\sigma}(q) \le r+2-b_{X(\CC)}(q)$.\\

\quad ({\bf b}) Assume $b_{X(\CC)}(q)>\lfloor (r+1)/2\rfloor$. Thus $r$ is even, $b_{X(\CC)}(q) = r/2 +1$ and $r_{X(\CC)}(q)=b_{X(\CC)}(q)$. To conclude the proof it is sufficient to find $A\in \Ss (X(\CC),q)$ with a label. Since $q\in \PP^r(\RR)$, we have $\sigma (\Ss(X(\CC),q)) = \Ss (X(\CC),q)$. Fix
$o\in X(\RR)$ and call
$\pi _o:
\PP^r(\CC)\setminus \{o\}\to \PP^{r-1}(\CC)$ the linear projection away from $o$. Since $o\in \PP^r(\RR)$, $\pi _o$ is defined over
$\RR$. Set $\Ss (X(\CC),o,q)= \{A\in \Ss (X(\CC),q)\mid o\in A\}$. The set $\Ss (X(\CC),o,q)$ is a Zariski closed
subset of the irreducible constructible set $\Ss (X(\CC),q)$. For a fixed $q$, it is clear we may take $o$ such that $\Ss
(X(\CC),o,q)\ne \emptyset$. For all
$A\in
\Ss (X(\CC),o,q)$ we have
$\sigma (A)\in
\Ss (X(\CC),o,q)$, because $o\in X(\RR)$. Let $X_o(\CC)\subset \PP^{r-1}(\CC)$ denote the closure of $\pi _o(X(\CC)\setminus
\{o\})$. The curve $X_o(\CC)$ is a rational normal curve. Since $\rho (X(\CC))=r+1$, for any $A\in \Ss (X(\CC),o,q)$
the set $A_o= \pi _o(A\setminus \{o\})$ is a linearly independent set with cardinality $r/2$ spanning $\pi
_o(q)$. Since $\sharp (A_o)=r/2 = \lfloor (r+1)/2 \rfloor$, Sylvester's theorem gives $r_{X_o(\CC)}(\pi _o(q)) = r/2 = b_{X_o(\CC)}(\pi _o(q))$ and $A_o\in \Ss(X_o(\CC),\pi_o(q))$. Therefore we are reduced to the case ({\bf a}) above: hence there exists a unique $A_o\in \Ss (X_o(\CC),\pi _o(q))$ and so it has a label. 

Now, suppose $A_o$ contains the point $X_o(\CC)\setminus \pi _o(X(\CC)\setminus \{o\})$ (i.e., the point
$\pi _o(T_oX(\CC)\setminus \{o\})$, where $T_oX(\CC)$ denote the tangent line to $X(\CC)$ at $o$). Thus $\pi_o^{-1}(A_o) = 2o \cup F_{o}\in \overline{\Ss(X(\CC), q)}\cong \PP^1(\CC)$, where $F_{o}$ is a reduced scheme. The general point of $\overline{\Ss(X(\CC), q)}$ is a reduced scheme and $X(\RR)$ is Zariski dense in $X(\CC)$. Hence there must exists $o'\in X(\RR)$ such that
$A_{o'}$ does not contain the point $X_{o'}(\CC)\setminus \pi _{o'}(X(\CC)\setminus \{o'\})$. Then there exist a unique $F_{o'}\subset X(\CC)\setminus\{o'\}$ such that $\pi_{o'}(F_{o'}) = A_{o'}$.  Since $A_{o'}$ has a label and $\pi _{o'}$ is defined over $\RR$, the uniqueness of $F_{o'}$ implies $\sigma (F_{o'}) =F_{o'}$. Thus $F_{o'}\cup \{o'\}$ has a label and $F_{o'}\cup \{o'\}\in \Ss (X(\CC),q)$.
\end{proof}

We recall that all integers between $\lfloor (r+2)/2\rfloor$ and $r$ are typical for the real degree $r$ rational normal curve; see \cite[Theorem 2.4]{b}. 
The next Corollary \ref{o2} shows a key difference between typical labels and typical ranks. It is a consequence of Theorem \ref{o1} and also 
of \cite[Theorem 1]{bb}. 

\begin{corollary}\label{o2}
Fix an integer $r\ge 2$. Let $X(\CC) \subset \PP^r(\CC)$ denote the degree $r$ rational normal curve with $X(\RR)\simeq
\PP^1(\RR)$. A label is typical for $X(\CC)$ if and only if it  has weight $\lfloor (r+2)/2\rfloor$.
\end{corollary}
\begin{proof}
Let $U$ denote the set of all $q\in \PP^r(\RR)$ such that $r_{X(\CC)}(q) =\lfloor (r+2)/2\rfloor$. By Sylvester's theorem
$\PP^r(\RR)\setminus U$ is the real part of a proper algebraic subvariety of $\PP^r(\CC)$ defined over $\RR$.
Thus $\PP^r(\RR)\setminus U$ is not Zariski dense in $\PP^r(\RR)$. By Theorem \ref{o1}, one has $\ell _{X(\CC),\sigma}(q)=
\lfloor (r+2)/2\rfloor$ for all $q\in U$: only labels of weight $\lfloor (r+2)/2\rfloor$ may be typical. 

The converse follows from Proposition \ref{typicallabels}. 
 \end{proof}

\begin{theorem}\label{o3}
Fix an even integer $r\ge 2$. Let $X(\CC)\subset \PP^r(\CC)$ be the linearly normal embedding of the plane curve $\Cc =
\{x^2+y^2+z^2=0\}\subset \PP^2(\CC)$. The complex curve $X(\CC)$ is a rational normal curve equipped with a real structure
such that $X(\RR)=\emptyset$. 
\begin{enumerate}
\item[(i)] For each $q\in \PP^r(\RR)$ we have $\ell _{X(\CC),\sigma}(q) = r_{X(\CC)}(q)$, $cr_{X(\CC),\sigma}(q) =
cr_{X(\CC)}(q)$ and these integers are even.

\item[(ii)] Every typical label has weight $r/2+1$.
\end{enumerate}
\end{theorem}

\begin{proof}
(i). It is clear from definitions that $\ell _{X(\CC),\sigma}(q) \ge r_{X(\CC)}(q)$ and $cr_{X(\CC),\sigma}(q) \ge
cr_{X(\CC)}(q)$. Since $X(\RR)=\emptyset$, each label has even weight. Thus, it is sufficient to prove that 
$\ell _{X(\CC),\sigma}(q) = r_{X(\CC)}(q)$, $cr_{X(\CC),\sigma}(q) =
cr_{X(\CC)}(q)$. Since $X(\CC)$ is an even degree rational normal curve (as a complex curve), we have
$b_{X(\CC)}(q) = cr_{X(\CC)}(q)$. We split the rest of the proof according to whether $b_{X(\CC)}(q)\leq r/2$ or $b_{X(\CC)}(q)=r/2+1$: 

\quad ({\bf a}) Assume $b_{X(\CC)}(q)\le r/2$. This is proven as case ({\bf a}) in the proof of Theorem \ref{o1}.

\quad ({\bf b}) Assume $b_{X(\CC)}(q) = r/2 +1$. First assume $r=2$, i.e., $X(\CC) =\Cc$. In this case, it is sufficient to take $D\cap X(\CC)$
where $D\subset \PP^2(\CC)$ is a line defined over $\RR$, containing $q$ and transversal to the smooth conic $\Cc$ (the latter
condition excludes only two lines). Now assume $r\ge 4$ and that the statement is true in lower dimensional projective spaces.
Fix $o\in X(\CC)$. Set $L= \langle o,\sigma (o)\rangle _{\CC}$. Note that the line $L$ is defined over $\RR$. Let
$\pi _L: \PP^r(\CC)\setminus L\to \PP^{r-2}(\CC)$ denote the linear projection away from $L$. Let $X_L(\CC)\subset \PP^{r-2}(\CC)$
denote the closure of $\pi _L(X(\CC)\setminus \{o,\sigma (o)\})$; $X_L(\CC)$ is a rational normal curve with a real structure
without real points. Mimic step ({\bf b}) of the proof of Theorem \ref{o1} to derive the result using the inductive assumption.

A proof similar to the one of Corollary \ref{o2} shows (ii).\end{proof}

\begin{remark}
The equality of admissible and complex ranks shown in Theorem \ref{o1} and Theorem \ref{o3} does not hold in general. 
In \cite[Example 2]{bb}, the authors gave an example of a homogeneous polynomial whose admissible rank (with respect to a Veronese variety) is strictly bigger than its complex rank. They also constructed an example of a curve such that a point of complex generic rank has higher admissible rank \cite[Example 1]{bb}; cf. with our Theorem  \ref{v1}. \end{remark}

\begin{definition}\label{boundary}
Let $X(\CC)\subset \PP^r(\CC)$ be a projective irreducible nondegenerate variety defined over $\RR$. Let $W(a,b)$ be the subset of all  $q\in \PP^r(\RR)$ such that $\ell_{X(\CC),\sigma}(q) = 2a+b$. Moreover, suppose that, for each typical label $(a,b)$ of $X(\CC)$, each $q\in W(a,b)$ admits only $(a,b)$ as a label. 

Let $(a,b)$ be a typical label of $X(\CC)$. Let $\mathcal W(a,b)=W(a,b)^{\circ}$ denote the interior of $W(a,b)$. The {\it topological boundary} $\partial(\mathcal W(a,b))$ is the set-theoretic difference between the closure of $\mathcal W(a,b)$ and the interior of the closure of $\mathcal W(a,b)$; see \cite[\S 3.1]{BrSt} or \cite[\S 1]{mmsv}. The {\it admissible rank boundary} $\partial_{\textnormal{alg}}(\mathcal W(a,b))$ is the complex Zariski closure of its topological boundary. This is a complex hypersurface in $\PP^r(\CC)$. 
\end{definition}

\begin{example}
Let $r=3$ and $X(\CC)\subset \PP^3(\CC)$ be the degree $3$ rational normal curve. It is a direct computation to see that all the points $q\in \langle Z\rangle$, where $Z\subset X(\CC)\setminus X(\RR)$ with $Z = \sigma(Z)$ and $\sharp(Z) = 2$, are binary cubics with only real roots; this implies that for every such $q$, 
$r_{X(\RR)}(q) = 3$ and $r_{X(\CC)}(q) = 2$. One can check that the tangential $\tau(X(\CC))$ is the locus of points whose real and complex ranks both equal to $3$ and coincides with the {\it real rank boundary} dividing real rank three points from real rank two points; cf. \cite[Remark 5.3]{v}.

By Corollary \ref{o2}, the typical labels are $(1,0)$ and $(0,2)$. Every $q\in W(1,0)$ satisfies $r_{X(\RR)}(q) = 3$ and every $p\in W(0,2)$ satisfies $r_{X(\RR)}(p) = 2$; they also have a unique label. Thus $\tau(X)$ coincides with the complex surface $\partial_{\textnormal{alg}}(\mathcal W(1,0))$.
\end{example}

\begin{remark}\label{component}
Let $X=X(\mathbb C)\subset \mathbb P^r(\mathbb C)$ be the rational normal curve, and let $g=\lfloor (r+2)/2\rfloor$ be the generic complex rank. 
Corollary \ref{o2} shows that the only typical labels $(a,b)$ are those whose weight is $2a+b=g$. 

For $r$ odd, we show that the hypersurface $J(\sigma_{g-2}(X),\tau(X))\subset \mathbb P^r(\mathbb C)$ is a component of the admissible rank boundary $\partial_{\textnormal{alg}}(\mathcal W(a,b)))$. To show this, we exhibit two sequences of points $p_{\epsilon}\in W(a,b)$ and $p'_{\epsilon}\in W(a-1,b+2)$ 
limiting to the same point $q\in J_{\RR}(\sigma_{g-2}(X)(\RR),\tau(X(\RR)))$. 

Fix $p\in W(a-1,b)$ and, for $\epsilon \in \mathbb R$, let
\[
\ell_\epsilon = \frac{1}{\epsilon}\left((iu+\epsilon v)^r + (-iu+\epsilon v)^r\right)\in \sigma_2(X), \mbox{ where } u,v\in X(\mathbb R),
\]
so that $p_{\epsilon} = p + \ell_\epsilon \in W(a,b)$. 

As above take the same $p\in W(a-1,b)$ and, for $\epsilon \in \mathbb R$, let
\[
\ell'_\epsilon = \frac{1}{\epsilon}\left((u+\epsilon v)^r - u^r\right) \in \sigma_2(X), \mbox{ where } u,v\in X(\mathbb R),
\]
so that $p'_{\epsilon} = p + \ell'_\epsilon \in W(a-1,b+2)$. 
Note that for $\epsilon\rightarrow 0$ both $\ell_{\epsilon}$ and $\ell'_{\epsilon}$ converge in projective space  to the point $u^{r-1}v\in \tau(X(\RR))$. 

Therefore $p_{\epsilon}$ and $p'_{\epsilon}$ converge to the same point $q\in J_{\RR}(\sigma_{g-2}(X)(\RR),\tau(X(\RR)))$. The set of such $q$ contains a Zariski open set of $J_{\RR}(\sigma_{g-2}(X)(\RR),\tau(X(\RR)))$ and so its complex Zariski closure is a component of the admissible rank boundary. The closure of this real join is the hypersurface $J(\sigma_{g-2}(X),\tau(X))\subset \mathbb P^r(\CC)$. 
\end{remark}

We conclude this section with an explicit example of a point having more than one label, i.e., a point having distinct {\it minimal} admissible decompositions with  distinct labels. Already when $X(\CC)$ is an even degree rational normal curve, a point may have more than one label with the same minimal weight: 

\begin{example}\label{twolabels}
Let $\CC[x,y]_4$ be the homogeneous coordinate ring of $\PP^4(\CC)$ where the rational normal curve of degree four $X(\CC)$ is embedded.

Let $q = \frac{3}{5}x^4+\frac{7}{5}x^3y+\frac{4}{3}x^2y^2+\frac{5}{4}xy^3+y^4\in \CC[x,y]_4$. Its apolar ideal is $\textnormal{Ann}(q) = (12915x^2y-29088xy^2+6220y^3,1435x^3-5652xy^2+1264y^3)$; see \cite[Definition 1.11]{IK} for background on apolarity and apolar ideals. The pencil $\phi(x,y,\lambda) = 12915x^2y-29088xy^2+6220y^3 + \lambda(1435x^3-5652xy^2+1264y^3) \subset \textnormal{Ann}(q)$ gives the $\PP^1$ of rank (real or complex) decompositions of $q$. We consider the univariate polynomial $\tilde{\phi}(t) = \phi(t,1,\lambda)$ and compute with \texttt{Macaulay2} \cite{M2} its discriminant $\mathrm{Disc}(\phi)$, which is the following polynomial in $\lambda$:
\[
\mathrm{Disc}(\phi)(\lambda) = -1359731348267443200\lambda^5-18287158078605830400\lambda^4+
 \]
 \[
-88301578772786601600\lambda^3-179185496017480948800\lambda^2-125609767833135474000\lambda.
\]
This has five real roots and divides $\PP^1(\RR)$ (the local coordinate being $\lambda$) into a union of intervals according to the behavior of the roots of  $\tilde{\phi}(t)$. On each of these connected components, the number of real roots of $\tilde{\phi}(t)$ is constant. One can check that there are intervals where $\tilde{\phi}(t)$ has three distinct reals roots and intervals where it has one real and two complex conjugate roots. Since the roots of $\tilde{\phi}(t)$ correspond to the summands of a $X(\CC)$-rank decomposition of $q$, the real point $q$ has both labels $(1,1)$ and $(0,3)$. 
\end{example}

\section{Real joins and typical labels}

We now shift gears to {\it real joins}. To introduce them, we equip $\PP^r(\RR)$ and $\RR^{r+1}$ with the usual Euclidean topology. Let $A, B\subseteq \PP^r(\RR)$ be non-empty semialgebraic subsets. The \emph{strict real join}
$J^0_{\RR}(A,B)$ of $A$ and $B$ is the subset $A+B$, where $+: \RR^{r+1}\times \RR^{r+1}\to \RR^{r+1}$ denotes the addition (we denote with the same
symbol the induced map on $\PP^r(\RR)$). The \emph{real join} $J_{\RR}(A,B)$ of $A$ and $B$ is the closure in $\PP^r(\RR)$ of $J^0_{\RR}(A,B)$. By \cite[Propositions 2.2.7 and 2.2.2]{bcr}, $J^0_{\RR}(A,B)$ and $J_{\RR}(A,B)$ are semialgebraic. Note that $J_{\RR}(A,B)
=J_{\RR}(\overline{A},\overline{B})$.
\begin{remark}\label{eu2}
Let $A, B, B' \subset \PP^r(\RR)$ be semialgebraic subsets. Then: 
\begin{enumerate}
\item[(i)] We have $J^0_{\RR}(A,B) =J^0_{\RR}(B,A)$ and hence $J_{\RR}(A,B) =J_{\RR}(B,A)$.

\item[(ii)]  We have $J_{\RR}(A,B) =J_{\RR}(\overline{A},\overline{B})$.

\item[(iii)]  We have $J^0_{\RR}(A,B\cup B') =J^0_{\RR}(A,B)\cup J^0_{\RR}(A,B')$ and hence $J_{\RR}(A,B\cup B')
=J_{\RR}(A,B)\cup J_{\RR}(A,B)$. Hence we may use a finite decomposition of $A$ and $B$ into semialgebraic subsets
homeomorphic to hypercubes to determine the real joins.

\item[(iv)]  By real dimension count, $\dim J_{\RR}(A,B)\le \min \{r,\dim A+\dim B+1\}$.

\item[(v)]  Let $X(\CC), Y(\CC)\subset \PP^r(\CC)$ be integral projective varieties defined over $\RR$. It is clear that
the usual complex join satisfies $J(X,Y)\cap \PP^r(\RR) \supseteq J_{\RR}(X(\RR),Y(\RR))$, but often the strict inequality holds. (Example: $r=3$ and
$X(\CC)=Y(\CC)$ be the rational normal curve.)
\end{enumerate}
\end{remark}

\begin{lemma}\label{aaa2}
Fix a closed non-empty (not necessarily equidimensional) semialgebraic set $B\subseteq \PP^r(\RR)$. Assume the existence of a semialgebraic set $S\subseteq \PP^r(\RR)$ such that $\langle S\rangle _{\RR} =\PP^r(\RR)$ and $J^0_{\RR}(S, B) \subseteq B$. Then $B = \PP^r(\RR)$.
\end{lemma}

\begin{proof}
We first prove that $S\subseteq B$. Fix $x\in S$ and assume $x\notin B$. Since $B\ne \emptyset$, there is $y\in B$. By assumption and passing to closure, $B$ contains the real line $\langle \{x,y\}\rangle$. Hence $x\in B$, a contradiction. Since $S$ spans $\PP^r(\RR)$, there are $p_0,\dots ,p_r\in S$ linearly independent and spanning $\PP^r(\RR)$. We have shown $p_i\in B$ for all $i$. To finish the proof, it is sufficient to prove that $\langle \{p_0,\dots ,p_i\}\rangle _{\RR} \subseteq B$ for all $i=0,\dots ,r$. This is true for $i=0$, because $S\subseteq B$.
We use induction on $i$. Fix $i<r$ and assume $\langle \{p_0,\dots ,p_i\}\rangle _{\RR} \subseteq B$. Since $p_{i+1}\in S$, we have $J^0_{\RR}(\{p_{i+1}\}, B)
\subseteq B$ and in particular $J^0_{\RR}(\{p_{i+1}\},\langle \{p_0,\dots ,p_i\}\rangle _{\RR})\subseteq B$. Taking closure yields the statement. \end{proof}

To conclude, we prove a result that seems to have a similar flavor (although over $\RR$ and with real joins) of \cite[Theorem 3.1]{bhmt}, but weaker:

\begin{theorem}\label{eu1}
Assume $X(\RR)=\emptyset$. Let $g = r_{\gen}(X(\CC))$ and $\gamma = 2\lceil g/2\rceil$. Let $m$ be the maximum of all $\ell
_{X(\CC),\sigma}(q)$ for $q\in \PP^r(\RR)$. Then $J_{\RR}(W_{2,\sigma}(X(\CC)), W_{k,\sigma}(X(\CC)))\subseteq W_{k-2,\sigma}(X(\CC))\cup W_{k+2,\sigma}(X(\CC))$ for all
even $k$ such that $m\ge k\ge \gamma +2$. \end{theorem}

\begin{proof}
Since $X(\RR)=\emptyset$, all labels have even weight and they are of type $(a,0)$ with weight $2a$. Recall that all typical
weights are $\ge \gamma$ by Proposition \ref{typicallabels}.

Decompose the semialgebraic sets $W^0_{2,\sigma}(X(\CC))$ and $W^0_{k,\sigma}(X(\CC))$ into a finite union of semialgebraic sets $E_i$ and $F_j$ respectively, where each of them is homeomorphic to a hypercube (possibly of different Hausdorff dimension); see Remark \ref{eu0}. 

Fix $i,j$, and set $E = E_i, F = F_j$. Thus $\ell _{X(\CC),\sigma}(q)=2$, for all $q\in E$, and $\ell _{X(\CC),\sigma}(q)=k$,
for all $q\in F$. Consider the real join $J^0_{\RR}(E,F)$ and decompose it into a finite union of semialgebraic sets $G_k$, as before. 
Fix $k$, and set $G = G_k$. We show $G\subset W_{k-2,\sigma}(X(\CC))\cup W_{k+2,\sigma}(X(\CC))$. This would imply $J_{\RR}(E,F)\subset W_{k-2,\sigma}(X(\CC))\cup W_{k+2,\sigma}(X(\CC))$ and, upon taking the union, the desired inclusion $J_{\RR}(W_{2,\sigma}(X(\CC)), W_{k,\sigma}(X(\CC)))\subseteq W_{k-2,\sigma}(X(\CC))\cup W_{k+2,\sigma}(X(\CC))$ (unless $k=m$, as in the statement).

Fix an arbitrary $q\in G$, with $q\in \langle \{q_1,q_2\}\rangle _{\RR}$, and where $q_1\in E$ and $q_2\in F$. Notice that $q$ cannot have label $(a,0)$ with $a< (k-2)/2$, because $\ell_{X(\CC),\sigma}(q_2) = k$. We show that $q$ has $((k-2)/2,0)$ or $((k+2)/2,0)$ as one of its labels of minimal weight; we split the proof into cases according to the value of $k$: \\
\indent \quad ({\bf a}) Assume $k=m$. Suppose $q$ does not have label $((k-2)/2,0)$. Since $m$ is maximal, $q$ has label $(k/2,0)$.
Thus $J^0_{\RR}(E,F)\subseteq W_{k,\sigma}(X(\CC))$. Taking the union over the $F_j$ and closure, this yields $J_{\RR}(E,W_{k,\sigma}(X(\CC)))\subseteq W_{k,\sigma}(X(\CC))$. Since $E$ spans $\PP^r(\RR)$, Lemma \ref{aaa2} implies $W_{k,\sigma}(X(\CC)) =\PP^r(\RR)$. This is in contradiction with 
the assumption $k\ge \gamma +2$ and the statement of Proposition \ref{typicallabels}, i.e., $\gamma$ is a typical weight-label. \\
\indent \quad ({\bf b}) Assume $\gamma +2\le k \le m-2$. We show that each element of $G$ has $((k-2)/2,0)$ or $((k+2)/2,0)$ as one of its labels. Otherwise, suppose $(k/2,0)$ is a label of $q$. Then $J_{\RR}(E,F)\subseteq W_{k,\sigma}(X(\CC))$.
Thus $J_{\RR}(E, W_{k,\sigma}(X(\CC))) \subseteq W_{k,\sigma}(X(\CC))$. As before, since $E$ spans $\PP^r(\RR)$, Lemma \ref{aaa2} provides
a contradiction. 
\end{proof}

Dropping the assumption $X(\RR)=\emptyset$, with the same arguments as above, we obtain:

\begin{theorem}\label{eu4}
Let $g = r_{\gen}(X(\CC))$ and $\gamma = 2\lceil g/2\rceil$. Let $m $ be the maximum of all $\ell_{X(\CC),\sigma}(q)$,
$q\in\PP^r(\RR)$, with $q$ having a label $(a,0)$. Let $W_{2a}(a,0)$ be the closure of the set of all $q\in \PP^r$ with $\ell
_{X(\CC),\sigma}(q)=2a$ and $(a,0)$ as one of their labels. Fix an even integer $k$ such that $m \ge k\ge \gamma +2$. Then
$J_{\RR}(W_2(1,0), W_k(k/2,0))\subseteq W_{k-2}((k-2)/2,0)\cup W_{k+2}((k+2)/2,0)$. \end{theorem}

\begin{small}

\end{small}

\end{document}